\documentclass{amsart}
\usepackage{amssymb,latexsym,amsmath,amsthm}
\usepackage[mathscr]{eucal}
\newtheorem{theorem}{Theorem}[section]
\newtheorem{lemma}[theorem]{Lemma}
\newtheorem{proposition}[theorem]{Proposition}
\newtheorem{corollary}[theorem]{Corollary}
\numberwithin{equation}{section}

\renewcommand{\l}{\lambda}
\newcommand{\g}{{\rm g}}
\newcommand{\RR}{\ensuremath{\mathbb{R}}}
\newcommand{\CC}{\ensuremath{\mathbb{C}}}

\newcommand{\prtl}{\ensuremath{\partial}}

\newcommand{\hf}{\ensuremath{\frac{1}{2}}}

\newcommand{\thetabar}{\overline\theta}
\newcommand{\Rn}{{\mathbb R}^n}
\newcommand{\R}{{\mathbb R}}

\title[Strichartz estimates on manifolds with boundary]{Strichartz estimates for the wave equation on manifolds with boundary}

\thanks{The authors were supported by the National Science
Foundation, Grants DMS-0654415 and DMS-0099642.}

\author{Matthew D. Blair}
\address{Department of Mathematics, University of New Mexico, Albuquerque, NM 87131}\email{blair@math.unm.edu}

\author{Hart F. Smith}
\address{Department of Mathematics, University of Washington,
Seattle, WA 98195} \email{hart@math.washington.edu}

\author{Christopher D. Sogge}
\address{Department of Mathematics, Johns Hopkins University,
Baltimore, MD 21218} \email{sogge@jhu.edu}

\begin{document}

\maketitle

\section{Introduction}\label{S:intro}
Let $(M,\g)$ be a Riemannian manifold of dimension $n \geq 2$.
Strichartz estimates are a family of space time integrability
estimates on solutions $u(t,x): (-T,T) \times M \to \CC$ to the
wave equation
\begin{align}\label{E:cprob}
\prtl_{t }^2 u(t,x) - \Delta_\g u(t,x) &= 0\,,
& & u(0,x) = f(x)\,, & & \prtl_t u(0,x) = g(x)
\end{align}
where $\Delta_\g$ denotes the Laplace-Beltrami operator on
$(M,\g)$. Local homogeneous Strichartz estimates state that
\begin{equation}\label{E:freestz}
\|u\|_{L^p((-T,T); L^q(M))} \leq C\left(\|f\|_{H^\gamma(M)} +
\|g\|_{H^{\gamma-1}(M)}
\right)
\end{equation}
where $H^\gamma$ denotes the $L^2$ Sobolev space over $M$ of order
$\gamma$, and $2 \leq p \leq \infty$, $2 \leq q <\infty$ satisfy
\begin{equation}\label{E:stzpair}
\frac 1p + \frac nq = \frac n2-\gamma
\qquad\qquad 
\frac 2p + \frac {n-1}q \leq \frac{n-1}2
\end{equation}
Estimates involving $q=\infty$ hold when $(n,p,q) \neq
(3,2,\infty)$, but typically require the use of Besov spaces.

Strichartz estimates are well established on flat Euclidean space,
where $M=\RR^n$ and $\g_{ij}=\delta_{ij}$. In that case, one can
obtain a global estimate with $T=\infty$; see for example
Strichartz~\cite{strich77}, Ginibre and Velo~\cite{ginvelo95},
Lindblad and Sogge~\cite{lindsogge95}, Keel and
Tao~\cite{keeltao98}, and references therein.  However, for
general manifolds phenomena such as trapped geodesics and
finiteness of volume can preclude the development of global
estimates, leading us to consider local in time estimates.

If $M$ is a compact manifold without boundary, finite speed of
propagation shows that it suffices to work in coordinate charts,
and to establish local Strichartz
estimates for variable coefficient wave operators on $\RR^n$. Such
inequalities were developed for operators with smooth coefficients
by Kapitanski~\cite{kapitanski91} and
Mockenhaupt-Seeger-Sogge~\cite{mss93}.  In this context one has
the Lax parametrix construction, which yields the appropriate
dispersive estimates.  Strichartz estimates for operators with 
$C^{1,1}$ coefficients
were shown by the second author in~\cite{SmC2stz} and by Tataru
in~\cite{T3}, the latter work establishing the full range of local
estimates.  Here the issue is more
intricate as the lack of smoothness prevents the use of the
Fourier integral operator machinery.  Instead, wave packets or
coherent state methods are used to construct parametrices for the
wave operator.

In this work, we consider the establishment of Strichartz estimates on a
manifold with boundary, assuming that the solution satisfies
either Dirichlet or Neumann homogeneous boundary conditions.
Strichartz estimates for certain values of $p,q$ were established
by Burq-Lebeau-Planchon \cite{blp} using results from \cite{SmSoBdry}; 
our work expands the range of indices $p$ and $q$, and includes new
estimates of particular interest for the critical nonlinear wave equation
in dimensions 3 and 4.
Our main result concerning Strichartz estimates is the following.

\begin{theorem}~\label{T:homogstz} Let $M$ be a compact Riemannian
manifold with boundary.  Suppose $2 < p \leq \infty$, $2 \leq q
< \infty$ and $(p,q, \gamma)$ is a triple satisfying
\begin{align}\label{E:hompair}
\frac{1}{p} + \frac nq = \frac n2-\gamma & & \begin{cases} \frac
3p + \frac {n-1}q \leq \frac{n-1}2, & n \leq 4\\
\frac 1p + \frac 1q \leq \frac 12, & n \geq 4
\end{cases}
\end{align}
Then we have the following estimates for solutions $u$
to~\eqref{E:cprob} satisfying either Dirichlet or Neumann
homogeneous boundary conditions
\begin{equation}\label{E:homogstz}
\|u\|_{L^p([-T,T]; L^{q} (M))} \leq C\left(\,\|f\|_{H^\gamma(M)} +
\|g\|_{H^{\gamma-1}(M)}
\right)
\end{equation}
with $C$ some constant depending on $M$ and $T$.
\end{theorem}

A lemma of Christ-Kiselev \cite{CK} 
allows one to deduce inhomogeneous Strichartz
estimates from the homogeneous estimates. In the following corollary, 
$(r',s')$ are
the H\"older dual exponents to $(r,s)$, and the assumptions imply that
a homogeneous $(H^{1-\gamma},H^{-\gamma})\rightarrow L^rL^s$ holds.

\begin{corollary}~\label{T:inhomogstz} Let $M$ be a compact Riemannian
manifold with boundary. Suppose that the triples $(p,q,\gamma)$ and
$(r',s',1-\gamma)$ satisfy the conditions of Theorem \ref{T:homogstz}.
Then we have the following estimates for solutions $u$
to~\eqref{E:cprob} satisfying either Dirichlet or Neumann
homogeneous boundary conditions
\begin{equation*}
\|u\|_{L^p([-T,T]; L^{q} (M))} \leq C\left(\,\|f\|_{H^\gamma(M)} +
\|g\|_{H^{\gamma-1}(M)} + \|F\|_{L^r([-T,T]; L^s(M))}
\right)
\end{equation*}
with $C$ some constant depending on $M$ and $T$.
\end{corollary}

For details on the proof of Corollary \ref{T:inhomogstz} using Theorem
\ref{T:homogstz} and the Christ-Kiselev lemma we refer to Theorem 3.2
of \cite{SmSoGlobal}, which applies equally well to Neumann conditions.

By finite speed of propagation, our results also apply to
noncompact manifolds, provided that there is uniform control
over the size of the metric and its derivatives in appropriate coordinate
charts.  In particular, we obtain local in time Strichartz estimates for
the exterior in $\RR^n$ of a compact set with smooth boundary,
for metrics $\g$ which agree with the Euclidean metric outside a compact
set. In this case one can obtain global in time Strichartz estimates
under a nontrapping assumption. We refer to \cite{SmSoGlobal}
for the case of odd dimensions, and Burq \cite{Burq} and Metcalfe \cite{Met}
for the case of even dimensions.  See also \cite{HMSSZ}.

For a manifold with strictly geodesically-concave boundary,
the Melrose-Taylor parametrix yields the Strichartz
estimates, for the larger range of exponents in (1.3) (not including endpoints) 
as was shown in~\cite{SmSoCrit}.
If the concavity assumption is removed, however,
the presence of multiply reflecting geodesics and their limits,
gliding rays, prevent the construction of a similar parametrix.

Recently,  Ivanovici \cite{Iv} has shown that, when $n=2$, (1.5) cannot hold for the full range of exponents in (1.3).  Specifically, she showed that if $M\subset \R^2$ is a compact convex domain with smooth boundary then (1.5) cannot hold when $q>4$ if $2/p+1/q=1/2$.  It would be very interesting to determine the sharp range of exponents for (1.5) in any dimension $n\ge2$.

The Strichartz estimates of Tataru \cite{T3} for Lipschitz metrics
yield estimates in the boundary case, but with a
strictly larger value of $\gamma$. The approach of \cite{T3} involves
the construction of parametrices which apply over short time intervals
whose size depends on frequency. Taking
the sum over such sets generates a loss of
derivatives in the inequality.

These ideas influenced the development of the spectral cluster
estimates for manifolds with boundary appearing
in~\cite{SmSoBdry}.  Such estimates were established through
squarefunction inequalities for the wave equation, which control
the norm of $u(t,x)$ in the space $L^q(M; L^2(-T,T))$. These
spectral cluster estimates were used in the
work of Burq-Lebeau-Planchon~\cite{blp} to establish Strichartz
estimates 
for a certain range of triples $(p,q,\gamma)$.
The range of triples that can be obtained in this manner, however, 
is restricted by the
allowed range of $q$ for the squarefunction estimate. In dimension
3, for example, this restricts the indices to $p,q\ge 5$. 
In~\cite{blp} similar estimates involving $W^{s,q}$ spaces were
also established, and used in conjunction with the Strichartz estimates
and boundary trace arguments
to establish global well-posedness for the critical semilinear
wave equation for $n=3$.  In the last two sections of this paper we shall present some new results
concerning critical semilinear wave equations.  Specifically, we shall obtain local well-posedness and global existence for small data when $n=4$, as well as a natural scattering result for $n=3$.

The approach of this paper instead adapts the proof of the
squarefunction inequalities in~\cite{SmSoBdry}. We utilize the
parametrix construction of that paper, and establish
the appropriate time-dispersion bounds on the associated
kernel. This allows us to obtain the Strichartz
estimates for a wider range of triples, including, for example,
the important $L^4((-T,T);L^{12}(M))$ estimate in dimension 3,
and the $L^3((-T,T);L^6(M))$ estimate in dimension 4.

The key observation in~\cite{SmSoBdry}
is that $u$ satisfies better estimates if it is
microlocalized away from directions tangent to $\prtl M$ than if
it is microlocalized to directions nearly tangent to $\prtl M$. This is
due to the fact that one can construct parametrices
over larger time intervals as one moves to
directions further away from tangent to $\prtl M$. More precisely,
the parametrix for directions at angle $\approx\theta$ away
from tangent to $\prtl M$ 
applies for a time interval of size $\theta$, which would
normally yield a $\theta$-dependent loss in the estimate.  However,
this loss can be countered by the fact that such directions
live in a small volume cone in frequency space.
For {\it sub-critical} estimates, i.e.~ where strict inequality
holds in the second condition in \eqref{E:stzpair}, this frequency
localization leads to a gain for small $\theta$. The restriction
on $p$ and $q$ in Theorem \ref{T:homogstz} arises from requiring
this gain to counteract the loss from adding over the $\theta^{-1}$
disjoint time intervals on which one has estimates.
Hence, while the range of $p$ and $q$
in our theorem is not known to be optimal, the restrictions
are naturally imposed by the local nature of the parametrix construction
in~\cite{SmSoBdry}.

\subsection*{Notation} The expression $X \lesssim Y$ means
that $X \leq CY$ for some $C$ depending only on the manifold,
metric, and possibly the triple $(p,q,\gamma)$ under
consideration. Also, we abbreviate $L^p(I;L^q(U))$  
by $L^p L^q(I \times U)$.

\section{Homogeneous Strichartz Estimates}\label{S:estimates}
The proof of Strichartz estimates is a direct adaptation of
the proof of squarefunction estimates in \cite{SmSoBdry}. The difference
is that Strichartz estimates result from time decay of
the wave kernel, whereas squarefunction estimates result from decay
with respect to spatial separation. Consequently, in \cite{SmSoBdry} the
wave equation was conically localized in frequency so as to become
hyperbolic with respect to a space variable labelled $x_1$, and
the equation factored so as to make $x_1$ the evolution parameter. 

In order to maintain the convention that 
$x_1$ is the evolution parameter, in this section
we set $x_1=t$, and will use $x'=(x_2,\ldots,x_{n+1})$ to denote spatial
variables in $\RR^n$. Thus $x=(x_1,x')$ is a variable on $\RR^{1+n}$.

We work in a geodesic-normal coordinate patch near $\partial M$ 
in which $x_n\ge 0$ equals distance to the boundary (the estimates
away from $\partial M$ follow from \cite{kapitanski91} and \cite{mss93}).
The coefficients of the metric $\g_{ij}(x')$ are extended to $x_n<0$
in an even manner, and the solution $u(x)$ is extended evenly in the
case of Neumann boundary conditions, and oddly in $x_n$ in case of Dirichlet
conditions. The extended solution then solves the extended wave equation
on the open set obtained by reflecting the coordinate patch in $x_n$.

Setting 
 $a^{11}(x)=\sqrt{\det g_{ij}(x')}$, 
we now work with an equation
$$
\sum_{i,j=1}^{n+1}D_i a^{ij}(x')D_j u(x)=0
$$
on an open set symmetric in $x_n$.
A linear change of coordinates, and shrinking the
patch if necessary, reduces to considering coefficients $a^{ij}(x)$ which are
pointwise close to
the Minkowski metric on the unit ball in $\RR^{1+n}$, 
and defined globally so as to equal that
metric outside the unit ball.

Following \cite[\S 2]{SmSoBdry},
the solution $u$ is then localized in frequency to a conic set where
$|\xi'|\approx|\xi_1|$. On the complement of this set the operator
is elliptic, and the Strichartz estimates follow from elliptic
regularity and Sobolev embedding. As in section 7 of \cite{SmSoBdry}, one
uses the fact that the coefficients are smooth in all variables but $x_n$,
and Sobolev embedding can be accomplished
using at most one derivative in the $x_n$ direction. 

The next step is to take a Littlewood-Paley dyadic decomposition 
$u=\sum_{k=1}^\infty u_k$ with $\widehat u_k$ 
localized in frequency to shells $|\xi'|\approx 2^k$. 
One lets $a^{ij}_k(x)$ denote the coefficients frequency localized in
the $x'$ variables to $|\xi'|\le 2^k$, 
and factorizes
$$
\sum_{i,j=1}^{n+1}a^{ij}_k(x)\xi_i\xi_j=a^{11}_k(x)
\bigl(\xi_1+p_k(x,\xi')\bigr)\bigl(\xi_1-p_k(x,\xi')\bigr)\,,
$$
where $p_k(x,\xi')\approx |\xi'|$.
Just as in [25, \S 2], (and the higher dimensional modifications in [25, \S 7]),
Theorem \ref{T:homogstz}
is reduced to establishing,
uniformly over $\l=2^k$, bounds of the form
\begin{equation}\label{E:theorem1'}
\|u_\l\|_{L^p_{x_1}L^q_{x'}(|x|\le 1)}\lesssim 
\l^\gamma\bigl(\|u_\l\|_{L^\infty L^2}+\|F_\l\|_{L^2}\bigr)\,,\quad
D_1 u_\lambda-P_\l(x,D')u=F_\lambda\,.
\end{equation}
Here, $P_\l(x,D')=\frac 12p_\l(x,D')+\frac 12p_\l(x,D')^*$, and
the symbol $p_\l(x,\xi')$ can be taken frequency localized in $x'$ frequencies
to $|\xi'|\le \l$, and $p_\l(x,\xi')=|\xi'|$ if $|\xi'|\not\approx\l$.

The setup is now the same as in \cite{SmSoBdry}, and the reductions of
\S3-\S6 of that paper, specifically their $n$-dimensional analogues
of \S7, apply directly. 
This starts with a 
decomposition $u_\l=\sum_j u_j$ corresponding to a
dyadic decomposition of $\widehat u_\l(\xi)$
in the $\xi_n$ variable to regions $\xi_n\in [2^{-j-2}\l,2^{-j+1}\l]$
where $\l^{-1/3}\le 2^{-j}\le 1$.

If $2^{-j}\ge \frac 18$, corresponding to non-tangential reflection,
then the estimates will follow as the case for $2^{-j}=\frac 18$, so
we restrict attention to the case $2^{-j}\le \frac 18$. Since 
$|\xi'|\approx\l$, this implies
that some remaining variable is $\approx\l$, and after rotation
we assume that $\widehat u_j(x_1,\xi')$ is supported in a set
$$
\{\xi:\xi_{n+1}\approx\l,\;
|\xi_j|\le c\l,\, j=2,\ldots,n-1,\;\mathit{and}\;\xi_n\approx\theta_j\l\}
$$
where $\l^{-1/3}\le\theta_j\le \frac 18$.

The proof establishes good bounds on the term $u_j$ over time intervals of
length $\theta_j$. Precisely, let $S_{j,k}$, 
$|k|\le\theta_j^{-1}$,
denote the time slice
$x_1\in[k\epsilon\theta_j,(k+1)\epsilon\theta_j]$.
In analogy with \cite[Theorem 3.1]{SmSoBdry}, we establish the bound
\begin{equation}\label{E:theorem1''}
\|u_j\|_{L^p_{x_1}L^q_{x'}(S_{j,k})}\lesssim
\l^\gamma\theta_j^{\sigma(p,q)}c_{j,k} 
\end{equation}
where $c_{j,k}$ satisfies the nested summability condition 
\cite[(3.1)]{SmSoBdry}, and where
$$
\sigma(p,q)=\begin{cases}
(n-1)(\tfrac 12-\tfrac 1q)-\tfrac 2p\,,
\quad&
(n-2)(\tfrac 12-\tfrac 1q) \le \tfrac 2p\\
\tfrac 12 - \tfrac 1q\,, 
&
(n-2)(\tfrac 12-\tfrac 1q) \ge \tfrac 2p
\end{cases}
$$
Adding over the $\theta_j^{-1}$ disjoint slabs intersecting $|x_1|\le 1$,
the simple uniform bounds on $c_{j,k}$ yield
$$
\|u_j\|_{L^p_{x_1}L^q_{x'}(|x_1|\le 1)}\lesssim \l^\gamma
\theta_j^{\sigma(p,q)-1/p}
\bigl(\,\|u_\l\|_{L^\infty L^2}+\|F_\l\|_{L^2}\bigr)\,.
$$
The $\theta_j$ take on dyadic values less than 1, and provided 
$\sigma(p,q)> 1/p$, one can sum over $j$ to obtain \eqref{E:theorem1'}.
In case $\sigma(p,q)=1/p$ one can also sum the series, using the
nested summability condition \cite[(3.1)]{SmSoBdry}, together with
the branching argument on \cite[page 118]{SmSoBdry},
to yield \eqref{E:theorem1'}.
Note that the restrictions on $(p,q)$ in Theorem \ref{T:homogstz}
are precisely that $\sigma(p,q)\ge 1/p$.

Estimate \eqref{E:theorem1''} is established through
the parametrix construction from \cite{SmC2spec}, together with the use
of the $V^p_2$ spaces of Koch-Tataru \cite{KT}. Precisely,
one rescales $\RR^{1+n}$ by $\theta_j$, and considers the symbol
$$
q(x,\xi')=\theta_j p_j(\theta_j x,\theta_j^{-1}\xi')\,,
$$
where $p_j$ is such that $\widehat q(x_1,\zeta,\xi')$ is supported in
$|\zeta|\le c\mu^{1/2}$, where $\mu=\theta_j\l$ is the
frequency scale at which $u_j(\theta_j x)$ is localized.
Fix $u(x)=u_j(\theta_jx)$ and $\theta_j=\theta$, where now
$1\ge\theta\ge\mu^{-1/2}$. One writes
$$
D_1 u-q(x,D')u=F+G
$$
where $G$ arises from the error term $(p-p_j)u_j$.
The bound \eqref{E:theorem1''} is a consequence of the following
bound (for a global $\varepsilon>0$)
\begin{multline}\label{E:mubound1}
\|u\|_{L^p_{x_1}L^q_{x'}(|x_1|\le \varepsilon)}\lesssim
\mu^{\gamma}\theta^{\sigma(p,q)}
\Bigl(\;\|u\|_{L^\infty L^2(S)}+
\|F\|_{L^1L^2(S)}\\
+\mu^{\frac 14}\theta^{\frac 12}
\|\langle \mu^{\frac 12}x_2\rangle^{-1}u\|_{L^2(S)}+
\mu^{-\frac 14}\theta_j^{-\frac 12}
\|\langle\mu^{\frac 12}x_2\rangle^2 G\|_{L^2(S)}
\,\Bigr)\,,
\end{multline}
and for $\theta=\mu^{-\frac 12}$
\begin{equation}\label{E:mubound2}
\|u\|_{L^p_{x_1}L^q_{x'}(|x_1|\le \varepsilon)}\lesssim
\mu^{\gamma}\theta^{\sigma(p,q)}
\Bigl(\;\|u\|_{L^\infty L^2(S)}+
\|F+G\|_{L^1L^2(S)}
\,\Bigr)\,.
\end{equation}

The solution $u$ is written as a superposition of terms,
each of which is product of $\chi_I(x_1)$, 
for an interval $I\subset[-\varepsilon,\varepsilon]$, with
a functions whose wave-packet transform is invariant under
the Hamiltonian flow of $q(x,\xi')$. The wave-packet transform,
which acts in the $x'$ variables, is a simple modification
of the Gaussian transform used by Tataru \cite{T3} to establish
Strichartz estimates for rough metrics; see also \cite{T}.
Precisely, set
$$
\bigl(T_\mu f\bigr)(x',\xi')=\mu^{n/4}\,
\int e^{-i\langle \xi',y'-x'\rangle}\,g\bigl(\mu^{\frac 12}(y'-x')\bigr)\,f(y')
\,dy'\,.
$$
The base function $g$ is taken to be of Schwartz class with $\widehat g$ 
supported in a ball of small radius. Thus, 
$\tilde u(x,\xi')=[T_\mu u(x_1,\cdot)](x',\xi')$ has the same
localization in $\xi'$ as does $\widehat u(x_1,\xi')$.

By Lemma 4.4 of \cite{SmSoBdry} one can write
$$
\Bigl(d_1-d_{\xi'} q(x,\xi')\cdot d_{x'}+d_{x'} q(x,\xi')\cdot d_{\xi'}\Bigr)
\tilde{u}(x,\xi')
=\tilde{F}(x,\xi')+\tilde{G}(x,\xi')\,.
$$
By variation of parameters and the use of $V^p_2$ spaces, one reduces
matters to establishing estimates for solutions invariant under the flow.
The use of the $V^p_2$ spaces from \cite{KT} requires $p>2$, which
is implied by the conditions of Theorem \ref{T:homogstz}.

Let $\Theta_{t,s}$ denote the Hamiltonian flow of $q(x,\xi')$,
from $x_1=s$ to $x_1=t$.
Then the bounds
\eqref{E:mubound1}-\eqref{E:mubound2} are consequences of the following,
which is the analogue of Theorem 7.2 of~\cite{SmSoBdry}.

\begin{theorem}\label{lastmaintheorem}
Suppose that $f\in L^2(\RR^{2n})$ is supported in
a set of the form 
$$
\{\xi:\xi_{n+1}\approx\mu\,,|\xi_j|\le c\mu\,,j=2,\ldots,n-1, \;\text{and}\;\;
\xi_n\approx\theta\mu\}
$$
or
$$
\{\xi:\xi_{n+1}\approx\mu\,,|\xi_j|\le c\mu\,,j=2,\ldots,n-1, \;\text{and}\;\;
|\xi_n|\le\mu^{\frac 12}\}
$$
in case
$\theta=\mu^{-1/2}$.

If
$\,W\!f(x_1,x')=T^*_\mu\bigl[f\circ\Theta_{0,x_1}\bigr](x')\,,$ then
for admissible $(p,q,\gamma)$
$$
\|W\!f\|_{L^p_{x_1}L^q_{x'}(|t|\le\varepsilon)}
\lesssim\mu^{\gamma}\theta^{\sigma(p,q)}
\,\|f\|_{L^2(\RR^{2n})}\,.
$$
\end{theorem}
\noindent{\it Proof.}
The function $W\!f$ is frequency localized to 
$\xi_n \approx \theta$ and $|\xi| \approx \mu \theta$ 
(respectively $|\xi_n| \leq  \mu^{\frac 12} $ when $\theta
\approx \mu^{-\frac 12}$).  By duality, it suffices to show the
estimate
\begin{equation}\label{E:duality}
\|W W^*F\|_{L^p L^q} \lesssim
\mu^{2\gamma}\theta^{2\sigma(p,q)} \|F\|_{L^{p'}L^{q'}}.
\end{equation}
for $\xi'$-frequency localized $F$. We use $t$ and $s$ in place of
$x_1$ and $y_1$ for ease of notation. Then
the operator $W W^*$ applied to $\xi'$-localized $F$
agrees with integration against the kernel
$$
K(t,x';s,y') = \mu^{\frac n2}\!\!\int e^{i\langle \zeta, x'-z\rangle -
i\langle \zeta_{s,t}, y'-z_{s,t} \rangle }
g(\mu^\hf(x'-z))\,g(\mu^\hf(y'-z_{s,t}))\,\beta_\theta(\zeta)\,dz\, d\zeta
$$
where $(z_{s,t},\zeta_{s,t})=\Theta_{s,t}(z,\zeta)$. To align with
the notation that $x'=(x_2,\ldots,x_{n+1})$ denote the
space parameters, we take $\zeta=(\zeta_2,\ldots,\zeta_{n+1})$.
Then $\beta_\theta(\zeta)$ is a smooth cutoff to the set
$$
\{\zeta:\zeta_{n+1}\approx\mu\,,|\zeta_j|\le c\mu\,,j=2,\ldots,n-1, 
\;\text{and}\;\;
\zeta_n\approx\theta\mu\}
$$
(respectively $|\zeta_n|\le \mu^{\frac 12}$ in case $\theta=\mu^{-\frac 12}$.)

Analogous to \cite[(7.1)-(7.2)]{SmSoBdry},
we establish the inequalities
\begin{equation}\label{E:energest}
\left\| \int K(t,x';s,y') f(y')\,dy' \right\|_{L^2_{x'}} \lesssim
\|f\|_{L^2_{y'}}.
\end{equation}
and
\begin{equation}\label{E:dispest}
\left\| \int K(t,x';s,y') f(y')dy \right\|_{L^\infty_{x'}} \lesssim
\mu^n \theta\,(1+\mu|t-s|)^{-\frac{n-2}{2}}
(1+\mu\theta^2|t-s|)^{-\hf} \|f\|_{L^1_{y'}}
\end{equation}
Interpolation then yields that
\begin{multline*}
\left\| \int K(t,x';s,y') f(y')\,dy' \right\|_{L^q_{x'}}\\
\lesssim (\mu^n \theta)^{1-\frac 2q
}(1+\mu|t-s|)^{-\frac{n-2}{2}(1-\frac 2q)}(1+\mu
\theta^2|t-s|)^{-\hf(1-\frac 2q)}\|f\|_{L^{q'}_{y'}}
\end{multline*}
In the case $\frac{n-2}{2}(1-\frac 2q) \leq \frac 2p \leq
\frac{n-1}{2}(1-\frac 2q)$, the exponent in the third factor on the
right can be
replaced by $\frac{n-2}{2}(1-\frac 2q)- \frac 2p \leq 0$, showing
that
\begin{equation*}
\left\| \int K(t,x';s,y') f(y')\,dy' \right\|_{L^q_{x'}}\lesssim
\mu^{2\gamma }\theta^{2((n-1)(\frac 12-\frac 1q) -\frac 2p)
} |t-s|^{-\frac 2p} \|f\|_{L^{q'}_{y'}}
\end{equation*}
In the case $\frac{n-2}{2}(1-\frac 2q) \geq \frac 2p$, we can ignore the last
factor and obtain the bound
\begin{equation*}
\left\| \int K(t,x';s,y') f(y')\,dy \right\|_{L^q_{x'}}\lesssim
\mu^{2\gamma }\theta^{2(\frac 12-\frac 1q) } 
|t-s|^{-\frac 2p} \|f\|_{L^{q'}_{y'}}
\end{equation*}
In both cases, the Hardy-Littlewood-Sobolev inequality then
establishes~\eqref{E:duality}. 

The inequality~\eqref{E:energest} is estimate \cite[(7.1)]{SmSoBdry},
which follows from the fact that $T_\mu$ is an isometry and 
$\Theta_{t,s}$ is a
measure-preserving diffeomorphism.  Hence it suffices to
prove~\eqref{E:dispest}. As in \cite{SmSoBdry}, we consider two
cases.

In the case $\mu\theta^2|t-s|\ge 1$, we fix $\thetabar\le\theta$ 
so that $\mu\thetabar^2|t-s|=1$,
and decompose $\beta_\theta(\zeta)$ into a sum of cutoffs $\beta_j(\zeta)$, 
each of which
is localized to a cone of angle $\thetabar$ about some direction $\zeta_j$. 
The proof of \cite[Theorem 5.4]{SmSoBdry} yields that
$$
|K_j(t,x';s,y')|\,\lesssim\,
\mu^n\thetabar^{n-1}
\bigl(\,1+\mu\thetabar\,|y'-x'_{s,t,j}|\bigr)^{-N}\,,
$$
where $x'_{s,t,j}$ is the space component of $\Theta_{s,t}(x,\zeta_j)$.
For each fixed $(s,t)$ the $x'_{s,t,j}$ are
a $(\mu\thetabar)^{-1}$ 
separated set, and adding over $j$ yields the desired bounds, since
in this case
$$
\mu^n\thetabar^{n-1}=
\mu^{\frac {n+1}2}\,|t-s|^{-\frac{n-1}2}\lesssim
\mu^n \theta\,(1+\mu|t-s|)^{-\frac{n-2}{2}}
(1+\mu\theta^2|t-s|)^{-\hf}
$$

In case $\mu\theta^2|t-s|\le 1$, we let $\thetabar\ge\theta$ be given by
$$
\thetabar=\min\bigl(\,\mu^{-\frac 12}|t-s|^{-\frac 12}\,,\,1\,\bigr)\,.
$$
Following the proof of \cite[(7.2)]{SmSoBdry},
we set $\zeta''=(\zeta_2,\ldots,\zeta_{n-1},\zeta_{n+1})$,
and let $\beta_j$ be a partition of unity in cones of angle 
$\thetabar$ on $\RR^{n-1}$.
We then decompose
$$
\beta_\theta(\zeta)=\sum_j \beta_\theta(\zeta)\,\beta_j(\zeta'')\,,
$$
and let $K=\sum_j K_j$ denote the corresponding kernel decomposition. 

The arguments on page 152 of \cite{SmSoBdry} yield
$$
|K_j(t,x';s,y')|
\lesssim\mu^n\thetabar^{n-2}\,\theta\,
\bigl(\,1+\mu\thetabar\,|(y'-x'_{s,t,j})_{2,\ldots,n-1}|\,\bigr)^{-N}\,.
$$
The $x'_{s,t,j}$ are $(\mu\thetabar)^{-1}$
separated in the 
$(2,\ldots,n-1)$ variables as $j$ varies, and summing over $j$ yields
$$
|K(t,x';s,y')|
\lesssim\mu^n\theta\,\thetabar^{n-2}
\approx
\mu^n \theta\,(1+\mu|t-s|)^{-\frac{n-2}{2}}
(1+\mu\theta^2|t-s|)^{-\hf}\,.\qed
$$

\section{Applications to semilinear wave equations}\label{S:nlw}

As an application, we consider the following family of semilinear
wave equations with defocusing nonlinearity
\begin{align}\label{E:nlw}
\prtl_{t }^2 u - \Delta u + |u|^{r-1}u =0 & & (u, \prtl_t u)|_{t=0} =
(f,g) & & u|_{\prtl M } = 0,
\end{align}
or
\begin{align}\label{E:nlwn}
\prtl_{t }^2 u - \Delta u + |u|^{r-1}u =0 & & (u, \prtl_t u)|_{t=0} =
(f,g) & & \partial_\nu u|_{\prtl M } = 0,
\end{align}
We will be mostly interested in the range of exponents $r < 1
+\frac{4}{n-2}$ (energy subcritical) and $r = 1 +\frac{4}{n-2}$
(energy critical).

In the boundaryless case where $\Omega = \Rn$, the first results for the critical wave equation were obtained by Grillakis \cite{Grillakis}.  He showed that when $n=3$ there are global smooth solutions of the critical wave equation, $r=5$, if the data is smooth.  Shatah and Struwe \cite{ShatStru} extended his theorem by showing that there are global solutions for data lying in the energy space $H^1\times L^2$.  They also obtained results for critical wave equations in higher dimensions.

For the case of obstacles, the first results are due to Smith and Sogge \cite{SmSoCrit}.  They showed that Grillakis' theorem extends to the case where $\Omega$ is the complement of a smooth, compact, convex obstacle and Dirichlet boundary conditions are imposed, i.e. \eqref{E:nlw} for $r=5$.  Recently this result was extended to the case of arbitrary domains in $\Omega \subset \R^3$ and data in the energy space by Burq, Lebeau and Planchon \cite{blp}.  The case of nonlinear critical Neumann-wave equations in 3-dimensions, \eqref{E:nlwn}, was subsequently handled by Burq and Planchon \cite{bp}.

The proofs of the results for arbitrary domains in 3-dimensions used two new ingredients.  First, the estimates of Smith and Sogge \cite{SmSoBdry} for spectral clusters turned out to be strong enough to prove certain Strichartz estimates for the linear wave equations with either Dirichlet or Neumann boundary conditions.  Specifially, Burq, Lebeau and Planchon \cite{blp} showed that one can control the $L^5W_0^{\frac3{10},5}$ norm of the solution of (1.1) over $[0,1]\times \Omega$ in terms of the energy norm of the data, assuming that $\Omega$ is compact.  The other novelty was new estimates for the restriction of $u$ to the boundary, specifically Proposition 3.2 in \cite{blp} and Proposition 3.1 in \cite{bp}.  In the earlier case of convex obstacles and Dirichlet boundary conditions treated in \cite{SmSoCrit} such estimates were not necessary since for the flux arguments that were used to treat the nonlinear wave equation \eqref{E:nlw}, the boundary terms had a favorable sign.  We remark that by using the results in Theorem 1.1, we can simplify the arguments in \cite{blp} and \cite{bp} since we now have control of the $L^4_tL^{12}_x([0,1]\times \Omega)$ norms of the solution of (1.1) in terms of the energy norm of the data.  If this is combined with the aforementioned boundary estimates in \cite{blp} and \cite{bp} one can prove the global existence results in these papers by using the now-standard arguments that are found in \cite{SmSoCrit} for convex obstacles, and \cite{ShatStru}  and \cite{Sog} for the case where $\Omega=\R^3$.  In the next section we shall show how these $L^4_tL^{12}_x$ and the weaker $L^5_tL^{10}_x$ estimates can be used to show that there is scattering for \eqref{E:nlw} when $n=3$, $r=5$ and $\Omega$ is the compliment of a star-shaped obstacle.

Let us conclude this section by presenting another new result.  We shall show that the Strichartz estimates in Theorem 1.1 are strong enough to prove 
the following:

\begin{theorem}\label{exist}  Suppose that $\Omega\subset \R^4$ 
is a domain with smooth compact boundary.  If $1<r<3$ and $(f,g)\in (\dot H^1(\Omega)\cap L^{r+1}(\Omega))\times L^2(\Omega)$ then  \eqref{E:nlw} and \eqref{E:nlwn} have a unique global solution satisfying
$$u\in C^0\bigl([0,T]; \dot H^1(\Omega)\cap L^{r+1}(\Omega)\bigr)\cap C^1\bigl([0,T]; L^2(\Omega)\bigr)\cap L^3_tL^6_x\bigl([0,T]\times \Omega\bigr)$$
for every $T>0$.  If $r=3$ then the same result holds provided that the $(\dot H^1\cap L^4)\times L^2$ norm of $(f,g)$ is sufficiently small.
\end{theorem}

The local existence results follow from the fact that Theorem~\ref{T:homogstz} implies that if $(\partial_t^2-\Delta) v=F$ and $v$ has either Dirichlet or Neumann boundary conditions then for $0<T<1$ there is a constant $C$ so that
\begin{equation}\label{loc4}
\|v\|_{L^3_tL^6_x((0,T)\times \Omega)}\le C\Bigl(\, \|v(0,\, \cdot\, )\|_{H^1}+\|\partial_t v(0,\, \cdot\, )\|_{L^2}+\int_0^T\|F(s,\, \cdot\, )\|_2\, ds\, \Bigr).
\end{equation}
If $\Omega$ is the complement of a bounded set, then estimate (3.3)
holds with $H^1$ replaced by $\dot{H}^1$, as can be seen by combining
the estimates for the case of compact $\Omega$ with the global
Strichartz estimates on $\R^4$, and using finite propagation velocity.
Using this estimate the theorem follows from a standard convergent iteration argument with $u$ in the space$$X=C^0((0,T); \dot H^1(\Omega)\cap L^{r+1}(\Omega))\cap C^1((0,T); L^2(\Omega))\cap L^3_tL^6_x((0,T)\times \Omega),$$
and $T$ being sufficiently small depending on the $(\dot{H}^1\cap L^{r+1})\times L^2$ norm of the initial data $(f,g)$ of either (3.1) or (3.2) for $1<r<3$, and $T$ depending on the
data in the critical case $r=3$. For data of sufficiently small norm,
one can obtain existence for $T=1$ for the critical case $r=3$.
Together with energy conservation, the above yields global existence
for $1<r<3$, and global existence for small data for $r=3$.

The analog of \eqref{loc4} when $n=3$ involves $L^5_tL^{10}_x$ in the left.  As we mentioned before, a stronger inequality involving $L^4_tL^{12}_x$ is valid when $n=3$ by Theorem~\ref{T:homogstz}.  Any such corresponding improvement of \eqref{loc4} when $n=4$ would lead to a global existence theorem for arbitrary data for the critical case where $r=3$, but, at present, we are unable to obtain such a result.

\section{Scattering for star-shaped obstacles in $3$-dimensions}

We now consider solutions to the energy critical nonlinear wave
equation in 3+1 dimensions in a domain $\Omega=\RR^3 \setminus
\mathcal{K}$ exterior to a compact, nontrapping obstacle
$\mathcal{K}$ with smooth boundary
\begin{align}
\Box u(t,x) &= (\prtl^2_t - \Delta) u(t,x) = - u^5(t,x),\qquad
(t,x) \in \RR \times \Omega  \notag\\
u\big|_{\RR\times \prtl \Omega} &= 0\label{nlw}\\
(\nabla u(t,\cdot), \prtl_t u(t,\cdot))& \in L^2(\Omega) \qquad t
\in \RR  \notag
\end{align}
We restrict attention to real-valued solutions $u(t,x)$.\\

When $\mathcal{K}$ is a nontrapping obstacle, the estimates above,
combined with those of Smith and Sogge \cite{SmSoGlobal} (see also Burq \cite{Burq}, Metcalfe \cite{Met}) imply
the following estimate on functions $w(t,x)$ satisfying
homogeneous Dirichlet boundary conditions
\begin{multline}\label{globstz}
\|w\|_{L^5(\RR; L^{10}(\Omega))} + \|w\|_{L^4(\RR;
L^{12}(\Omega))}\\ \leq C\left( \|\left(\nabla_{x} w(0,\cdot),
\prtl_t w(0,\cdot)\right) \|_{L^2(\Omega)} + \|\Box
w\|_{L^1(\RR;L^2(\Omega))} \right).
\end{multline}
In this section, we show how these global estimates can be used to
show that solutions to the nonlinear equation~\eqref{nlw} above
scatter to a solution to the homogeneous equation
\begin{align}
\Box v(t,x) &= 0,\qquad
(t,x) \in \RR \times \Omega  \notag\\
v\big|_{\RR\times \prtl \Omega} &= 0 \label{homog}\\
(\nabla v(t,\cdot), \prtl_t v(t,\cdot))& \in L^2(\Omega) \qquad t
\in \RR  \notag.
\end{align}

Let $\nu = \nu(x)$ denote the outward pointing unit normal vector
to the boundary at $x \in \prtl \mathcal{K}$.  We call the
obstacle $\mathcal{K}$ \emph{star-shaped with respect to the
origin} if $\nu(x) \cdot x \geq 0$ for all $x \in \prtl
\mathcal{K}$.  Define the energy functional
$$
E_0(v; t) =\frac 12 \int_\Omega |\nabla_x v(t,x)|^2 + |\prtl_t
v(t,x)|^2\;dx,
$$
and recall that $t \mapsto E_0(v;t)$ is conserved whenever $v$ is
a solution to the homogeneous equation~\eqref{homog}.  We show the
following:
\begin{proposition}\label{T:scatter}
Suppose $u$ solves the nonlinear problem~\eqref{nlw} and that
$\mathcal{K}$ is star-shaped with respect to the origin. Then
there exists unique solutions $v_{\pm}$ to~\eqref{homog} such that
\begin{equation}\label{asymnrg}
\lim_{t\to \pm \infty } E_0(u-v_{\pm};t) = 0.
\end{equation}
Moreover, $u$ satisfies the space-time integrability bound
\begin{equation}\label{sptimeint}
\|u\|_{L^5(\RR; L^{10}(\Omega))} + \|u\|_{L^4(\RR;
L^{12}(\Omega))} < \infty.
\end{equation}
\end{proposition}
When $\mathcal{K}=\varnothing$, this follows from the observations
of Bahouri and G\'{e}rard~\cite{bahourigerard}.  We also remark
that when $\mathcal{K}$ is convex, similar results for compactly
supported, subcritical nonlinearities were obtained by Bchatnia
and Daoulatli~\cite{bchatnia}.\\

Attention will be restricted to the $v_+$ function, as symmetric
arguments will yield the existence of a $v_{-}$ asymptotic to $u$
at $-\infty$.  As observed in~\cite{bahourigerard}, we actually
have that~\eqref{asymnrg} follows as a consequence
of~\eqref{sptimeint}.  We first establish the existence of the
wave operator, 
namely that for any solution $v$ to~\eqref{homog}, there exists a
unique solution $u$ to~\eqref{nlw} such that
$$
\lim_{t \to \infty}E_0(u-v;t) = 0.
$$
Given~\eqref{globstz}, for any $\delta >0$ we may select $T$ large
so that $\|v\|_{L^5([T,\infty);L^{10}(\Omega))} \leq \delta$.
Given any $w(t,x)$ satisfying
$\|w\|_{L^5([T,\infty);L^{10}(\Omega))} \leq \delta$, we have a
unique solution to the linear problem
\begin{align*}
&\Box \tilde{w} = -(v+w)^5 \\
&\lim_{t\to \infty}E_0(\tilde{w};t) =0
\end{align*}
as the right hand side is in $L^1([T,\infty);L^2(\Omega))$.  The
estimate~\eqref{globstz} then also ensures that
$$
\|\tilde{w}\|_{ L^5([T,\infty);L^{10}(\Omega))} \leq C
\|v+w\|_{ L^5([T,\infty);L^{10}(\Omega))}^5 \leq 32C \delta^5.
$$
Hence 
for $\delta$ sufficiently small, the map $w \to \tilde w$ is seen to be a contraction on the
ball of radius $\delta$ in 
$L^5([T,\infty);L^{10}(\Omega))$.  The unique fixed point $w$ can
be uniquely extended over all of $\RR \times \Omega$. Hence taking
$u=v+w$ shows existence of the wave operator.\\

To see that the wave operator is surjective, we need a decay
estimate which establishes that the nonlinear effects of the
solution map for~\eqref{nlw} diminish as time evolves.
\begin{lemma}\label{T:l6decay}
Let $\mathcal{K}$ be star-shaped with respect to the origin.  If
$u(t,x)$ solves~\eqref{nlw}, then the following decay estimate
holds
$$
\lim_{t \to \infty} \frac 16 \int_\Omega |u(t,x)|^6 \;dx = 0.
$$
\end{lemma}

When $\mathcal{K} = \varnothing$, this is due to Bahouri and
Shatah~\cite{bahourishatah}.  The proof below is essentially
theirs, with slight modifications made to handle the boundary
conditions.  However, for the sake of completeness, we replicate
the full proof below.  
We remark that the approach has its roots in arguments of 
Morawetz~\cite{morawetz}, and is related to other wrks regarding the decay of local energy for linear solutions in domains exterior to a star-shaped obstacle.
\\

To see that this implies the proposition, observe that given any
$\varepsilon>0$, there exists $T$ sufficiently large such that
$$
\sup_{t \geq T} \|u(t,\cdot)\|_{L^6} < \varepsilon.
$$
Hence for any $S>T$ we obtain the following for any solution $u$
to~\eqref{nlw}
\begin{multline*}
\|u\|_{L^5([T,S]; L^{10}(\Omega))} + \|u\|_{L^4([T,S];
L^{12}(\Omega))} \leq C \left( E + \|u^5\|_{L^1([T,S];
L^{2}(\Omega))} \right) \\
\leq CE + C\varepsilon\|u\|_{L^4([T,S]; L^{12}(\Omega))}
\end{multline*}
where $E$ denotes the conserved quantity
$$
E = E(t) = \int_\Omega \frac 12 |\nabla u (t,x)|^2 + \frac 12
|\prtl_t u(t,x)|^2 + \frac 16 |u(t,x)|^6\;dx.
$$
A continuity argument now yields $\|u\|_{L^5([T,\infty);
L^{10}(\Omega))} + \|u\|_{L^4([T,\infty); L^{12}(\Omega))}< 2CE$
and by a time reflection argument,~\eqref{sptimeint} follows.
However, this implies that the linear problem
$$
\Box w = -u^5 \qquad \qquad \lim_{t\to \infty } E_0(w;t) = 0
$$
admits a solution, showing that the wave operator is indeed
surjective as $v=u-w$ is the desired solution to~\eqref{homog}.

\begin{proof}[Proof of Lemma~\ref{T:l6decay}]
By a limiting argument it suffices to consider smooth, classical
solutions $u$ which decay at infinity.  We must show that for for
any $\varepsilon_0>0$, there exists $T_0$ such that whenever $t
\geq T_0$,
$$ \frac{1}{6}\int_\Omega |u(t,x)|\;dx \leq \varepsilon_0.
$$
Consider the stress energy tensor associated with $u$ (see
Tao~\cite{tao}, p. 149)
\begin{align*}
T^{00} &= \frac 12 (\prtl_t u )^2 +\frac 12 |\nabla u|^2 + \frac
16 u^6\\
T^{0j} &= - \prtl_t u \prtl_{x_j} u & 1 \leq j \leq 3\\
T^{jk} &= \prtl_{x_j} u \prtl_{x_k} u
-\frac{\delta_{jk}}{2}(|\nabla u|^2-(\prtl_t u)^2 + \frac 13 u^6)
& 1 \leq j, k \leq 3.
\end{align*}
It can be checked that the divergence free property holds
\begin{equation*}
\prtl_t T^{00} + \prtl_{x_j} T^{0j} =0 \qquad \prtl_t T^{0j} +
\prtl_{x_k} T^{jk} =0
\end{equation*}
with the summation convention in effect.  Taking the first of
these identities and applying the divergence theorem to a region
$\{0 \leq t \leq T, |x| \geq R+t \}$ (with $R >0$ large enough so
that $\mathcal{K} \subset B_R(0)$) we have
\begin{multline}\label{smallflux}
\int_{|x| \geq R + T} \frac 12 |\prtl_t u(T,x)|^2+\frac 12 |\nabla
u(T,x)|^2 + \frac 16 |u(T,x)|^6\;dx + \frac{1}{\sqrt{2}} \text{
flux}(0,T)\\ \leq \int_{|x| \geq R} \frac 12 |\prtl_t
u(0,x)|^2+\frac 12 |\nabla u(0,x)|^2 + \frac 16 |u(0,x)|^6\;dx
\end{multline}
where
$$
\text{ flux}(a,b) := \int_{M^a_b} \frac
12\left|\frac{x}{|x|}\prtl_t u + \nabla u\right|^2 +
\frac{|u|^6}{6}\;d\sigma
$$
$$
M^a_b := \{ a < t< b, |x|=R+t\}
$$
Since the solution has finite energy, we may select $R$ large so
that the right hand side of~\eqref{smallflux} is less than
$\frac{\varepsilon_0}{40}$ (and again $\mathcal{K} \subset
B_R(0)$). By time translation, $t \mapsto t+R$, it will suffice to
show the existence of $T_0$ such that whenever $t >T_0$ we have
$$
\frac{1}{6}\int_{x\in \Omega : |x|\leq t} |u(t,x)|\;dx \leq
\frac{\varepsilon_0}{2}
$$
(the additional smallness in the right hand side
of~\eqref{smallflux} will be used later in the proof).\\

We now define the following vector field $X=(X^0,X^1,X^2,X^3)$ by
contracting the stress-energy tensor with the null vector field $t
\prtl_t - x \cdot \nabla_x$ and adding a correction term
\begin{align*}
X^0 &= tT^{00} - x_k T^{0k} + u \prtl_t u\\
X^j &= tT^{j0} - x_k T^{jk} - u \prtl_{x_j} u \qquad 1 \leq j \leq
3.
\end{align*}
The space-time divergence of $X$ satisfies
$$
\text{div}(X) = -\frac 13 u^6.
$$

We now apply the divergence theorem over the truncated cone
$K^{T_2}_{T_1}= \{x\in \Omega: \, |x| \leq t, T_1 \leq t \leq T_2 \}$
\begin{align*}
0&=\int_{D(T_2)}X^0\;dx - \int_{D(T_1)}X^0\;dx -
\int_{M_{T_1}^{T_2}} \left( X^0 - \sum_{j=1}^3 \frac{x_j}{|x|}X^j
\right) \;d\sigma\\ & + \int_{K_{T_1}^{T_2}} \frac{|u|^6}{3}\;dx\;
dt -\int_{\prtl \Omega} \nu \cdot \langle X^1, X^2, X^3
\rangle \;d\sigma\\
&= I + II + III + IV + V
\end{align*}
where $d\sigma$ denotes Lebesgue measure on the corresponding
surface and $D(T_i) = \{x \in \Omega : |x| \leq T_i\}$,
and $M_{T_1}^{T_2}=\{ \, |x|=t, T_1\le t\le T_2\}$.  The
star-shaped assumption is crucial in controlling the last term
$V$.  Indeed, consider the restriction of the integrand in $V$ to
the boundary $\prtl \Omega (= \prtl \mathcal{K})$ and observe that the
Dirichlet boundary condition gives
\begin{align*}
\nu \cdot \langle X^1, X^2, X^3 \rangle  &= -\sum_{1 \leq j,k \leq
3} \nu_j x_k \left( \prtl_{x_j} u \prtl_{x_k} u
-\frac{\delta_{jk}}{2}|\nabla u|^2 \right)\\
& = -\left( \nu \cdot \nabla u \right) \left( x \cdot \nabla u
\right) + \frac 12 (\nu \cdot x) |\nabla u|^2.
\end{align*}
We have that $\nabla u $ is normal to $\prtl \Omega$ and hence
$|\nabla u|^2=(\nu \cdot \nabla u)^2$.  Treating $x$ as a vector,
we can project it on to the subspace orthogonal to $\nu$ obtaining
$$
0 = \nabla u \cdot (x - (\nu \cdot x) \nu) = x \cdot \nabla u -
(\nu \cdot x )(\nu\cdot \nabla u).
$$
This now gives
$$
\nu \cdot \langle X^1, X^2, X^3 \rangle = - \frac 12 \left( \nu
\cdot x \right) \left( \nu \cdot \nabla u \right)^2 \leq 0
$$
and since $IV \geq 0$ is clear,
$$
0 \geq I + II + III.
$$

We now impose polar coordinates $(r,\omega) \in \RR \times
\mathbb{S}^2$ on the third term, writing
$$
III=-\frac{1}{\sqrt{2}} \int_{M_{T_1}^{T_2}} \left( r(\prtl_t u +
\prtl_r u)^2 + u(\prtl_t u + \prtl_r u ) \right)\;d\sigma
$$
where $\prtl_r = \frac{x}{|x|} \cdot \nabla$ denotes the radial
derivative.  Next parameterize $M_{T_1}^{T_2}$ by $(r,\omega) \to
(r,r\omega)$ and set $v(y) = u(|y|,y)$ (or $v(r\omega) =
u(r,r\omega)$ in polar coordinates) so that we may write concisely
\begin{align*}
III & = - \int_{\mathbb{S}^2}\int_{T_1}^{T_2} r \left(\prtl_r v +
\frac vr \right)^2 r^2\;dr\;d\omega +
\int_{\mathbb{S}^2}\int_{T_1}^{T_2} \frac 12 \prtl_r\left( r^2 v^2 \right)\;dr\;d\omega\\
& = - \int_{\mathbb{S}^2}\int_{T_1}^{T_2} r \left(\prtl_r v +
\frac vr \right)^2 r^2\;dr\;d\omega + \frac 12 \int_{\mathbb{S}^2}
T_2^2 v^2(T_2 \omega)\;d\omega - \frac 12 \int_{\mathbb{S}^2}
T_1^2 v^2(T_1 \omega)\;d\omega
\end{align*}

To handle the first term $I$, first observe that in polar
coordinates
\begin{equation*}
|\nabla u|^2 = (\prtl_r u)^2 + \frac{1}{r^2}|\nabla_\omega u|^2 =
(\prtl_r u + \frac 1r u )^2 + \frac{1}{r^2}|\nabla_\omega u|^2 -
\frac 1{r^2} \prtl_r (ru^2).
\end{equation*}
Since $\mathcal{K}$ is star-shaped we may parameterize $\prtl
\Omega$ by $(r,\omega)=(\Psi(\omega),\omega)$ where $\Phi$ is a
real valued function on $\mathbb{S}^2$. This allows us to write
\begin{align}
I &= \int_{D(T_2)} \frac{T_2}{2} \left( (\prtl_t u)^2 +
\left(\prtl_r u + \frac 1r u \right)^2 +
\frac{1}{r^2}|\nabla_\omega u|^2 + \frac 13 u^6 \right) +
r\left(\prtl_r + \frac 1r u\right) \prtl_t u\;dx\notag\\
&\phantom{=}-\frac 12 \int_{\mathbb{S}^2}
\int_{\Psi(\omega)}^{T_2} T_2 \prtl_r (ru^2)\;dr\;d\omega
\label{intI}
\end{align}
Integrating by parts in the last term yields cancellation with one
of the terms in $III$ as the boundary condition gives $-\frac 12
\int_{\mathbb{S}^2} \int_{\Psi(\omega)}^{T_2} T_2 \prtl_r
(ru^2)\;dr\;d\omega = - \frac 12 \int_{\mathbb{S}^2} T_2^2 v^2(T_2
\omega)\;d\omega $. Similarly,
\begin{align*}
II &= -\int_{D(T_1)} \frac{T_1}{2} \left( (\prtl_t u)^2 +
\left(\prtl_r u + \frac 1r u \right)^2 +
\frac{1}{r^2}|\nabla_\omega u|^2 + \frac 13 u^6 \right) +
r\left(\prtl_r + \frac 1r u\right) \prtl_t u\;dx\\
&\phantom{=}+\frac 12 \int_{\mathbb{S}^2} T_1^2 v^2(T_1
\omega)\;d\omega
\end{align*}

In order to control remaining term in $I$ we need to observe the
following Hardy inequality, which holds in the exterior domain
\begin{equation}\label{hardy}
\int_\Omega \frac{|u|^2}{|x|^2} \;dx \leq 4 \int_\Omega |\nabla
u|^2 \;dx.
\end{equation}
To see this, we assume $u$ is real-valued and denote the integral
on left hand side as $J$.  Converting to polar coordinates
$$
J=\int_{\mathbb{S}^2} \int_{\Psi(\omega)}^\infty (u(r\omega))^2
\;dr\;d\omega = \int_{\mathbb{S}^2} r
u(r\omega)^2\Big|_{\Psi(\omega)}^\infty \;d \omega -
\int_{\mathbb{S}^2} \int_{\Psi(\omega)}^\infty 2u ( \prtl_r u) r
\;dr\;d\omega.
$$
The first term on the right is nonpositive (provided $u$ exhibits
sufficient decay at infinity) and Cauchy-Schwartz on the second
term gives
$$
J \leq 2 \sqrt{J}\left( \int_{\mathbb{S}^2}
\int_{\Psi(\omega)}^\infty |\prtl_r u|^2 r^2
\;dr\;d\omega\right)^{\frac 12}.
$$
The inequality~\eqref{hardy} now follows.

We now observe that the first integral in~\eqref{intI} is bounded
below by $T_2 \int_{D(T_2)} \frac{|u|^6}{6}\;dx$. Setting $T_2=T
>0$ and $T_1 = \varepsilon T$ ($0<\varepsilon<1$) and using the Hardy
inequality~\eqref{hardy} to control the first integral in $II$ now
yields
\begin{equation*}
T \int_{D(T)} \frac{|u|^6}{6}\;dx \leq C \varepsilon T E +
\int_{\varepsilon T}^T \int_{\mathbb{S}^2} T\left(\prtl_r v +
\frac vr \right)^2 r^2 \;d\omega\;dr.
\end{equation*}
Here $E$ is the conserved quantity $E=E(t)=\int_\Omega
T^{00}(t,x)\;dx$.  We can now divide both sides of this inequality
by $T$ and choose $\varepsilon$ sufficiently small so that
$C\varepsilon E \leq \varepsilon_0/4$, leaving us to control the
integral involving $v$.  However, by the proof of the Hardy
inequality above we have
$$
\int_{\varepsilon T}^T \int_{\mathbb{S}^2} \left(\prtl_r v + \frac
vr \right)^2 r^2 \;d\omega\;dr \leq 10 \int_{\varepsilon T}^\infty
\int_{\mathbb{S}^2} \left(\prtl_r v \right)^2 r^2 \;d\omega\;dr
\leq 10 \text{ flux}(\varepsilon T,\infty) <
\frac{\varepsilon}{4},
$$
provided $T$ is large enough so that $\varepsilon T>R$.
\end{proof}


\end{document}